\documentclass[12pt]{amsart}

\usepackage{amsmath}
\usepackage{amsthm}
\usepackage{amssymb}
\usepackage{tipa}

\usepackage{appendix}

\usepackage{graphicx}
\usepackage[english]{babel}

\usepackage{xcolor}   
\usepackage{hyperref}

\usepackage[shortlabels]{enumitem}

\newtheorem{theorem}{Theorem}[section]
\newtheorem{lemma}{Lemma}[section] 

\newtheorem{prop}{Proposition}[section]
\theoremstyle{Definition}
\newtheorem{definition}{Definition}
\theoremstyle{Example}

\newtheorem{remark}{Remark}

\DeclareMathOperator{\A}{A}

\newcommand{\lra}{\longrightarrow}

\newcommand{\F}{\mathbb{F}}
\newcommand{\Z}{\mathbb{Z}}

\DeclareMathOperator{\Per}{Per}

\title{Transitive and non-transitive subgroups of permutation groups}

\author{Arda Demirhan}
\address{Arda Demirhan\\Department of Mathematics\\ University of Rochester\\
  Rochester, NY, 14620, USA}
\email{a.demirhan@rochester.edu}

\author{Jacob Miller}
\address{Jacob Miller\\Department of Mathematics\\ University of Iowa\\
 Iowa City, Iowa, 52242-1419, USA}
\email{jacob-miller-2@uiowa.edu}

\author{Yixu Qiu}
\address{Yixu Qiu\\Department of Mathematics\\ University of Rochester\\
  Rochester, NY, 14620, USA}
\email{yqiu21@u.rochester.edu}

\author{Thomas J. Tucker}
\address{Thomas J. Tucker\\Department of Mathematics\\ University of Rochester\\
  Rochester, NY, 14620, USA}
\email{tjtucker@gmail.com}

\author{Zheng Zhu}
\address{Zheng Zhu\\Department of Mathematics\\ University of Rochester\\
  Rochester, NY, 14620, USA}
\email{zzhu32@ur.rochester.edu}
\date{\today}
\begin{document}

\maketitle

\begin{abstract}
 We treat the problem of finding transitive subgroups $G$ of $S_n$
 containing normal subgroups $N_1$ and $N_2$, with $N_1$ transitive
 and $N_2$ not transitive, such that $G/N_1 \cong G/N_2$.  We show
 that such $G$ exist whenever $n$ has a prime factor that also divides
 $\phi(n)$.  We show that no such $G$ exist when $n = pq$ for $p < q$
 with $p$ not dividing $q-1$.  
\end{abstract}

\section{Introduction}
Suppose we have an $n$ such that $G \subseteq S_{n}$ contain two
normal subgroups $N_1$ and $N_2$ such that $G/N_1 \cong G/N_2$ where
$N_{1}$ is transitive and $N_{2}$ is not (acting on $\{1, \dots n \}$.
Zheng \cite{Zhu} has shown that one take can any such group and
create a group action on the infinite $n$-ary toted tree $T_n$ with
the curious property that the fixed point process for this group
action is not a Martingale, answering a question implicit in the work
of Bridy, Jones, Kelsey and Lodge \cite{Rafe, BJKL} on iterated Galois groups of
rational functions.

In this paper, we will a partial answer to the question of which $n$
admit permutation groups $G$ of degree $n$ that contain normal subgroups $N_1$ and $N_2$ such that $G/N_1 \cong G/N_2$ where
$N_{1}$ is transitive and $N_{2}$ is not.  

Recall that $\phi(n)$ (the Euler $\phi$-function of $n$) is the number
of positive integers less than $n$ that are prime to $n$.  

\begin{theorem}\label{example-thing}
  Let $n$ be an integer that has a prime factor $p$ such that $p$
  divides $\phi(n)$.  Then there is a transitive permutation group $G$
  of degree $n$ that contains a normal transitive subgroup $N_1$ of
  index $p$ and a normal non-transitive subgroup $N_2$ of index $p$.
\end{theorem}

While we cannot prove a true converse to Theorem \ref{example-thing},
we can prove the following.  
\begin{theorem}\label{main-thm}
  Let $p < q$ be primes.  Suppose that $p$ does not divide $q-1$.
  Then there is no permutation group $G$ of degree $pq$ with the
  property that $G$ contains a transitive subgroup $N_1$ and a
  non-transitive subgroup $N_2$ such that $G/N_1 \cong G/N_2$.
\end{theorem}

An outline of the paper is as follows.  In the next section, we prove
Theorem \ref{example-thing}.  In Section \ref{wreath}, we introduce
wreath products, which are crucial in our proof of Theorem
\ref{main-thm}.  , we obtain an embedding of our $G$ into
$S_p \wr S_q$, described in Proposition \ref{theta}. This essentially
reduces our problem to the study of transitive subgroups of $S_q$.  In
Section \ref{perm}, we present known results on permutation groups of
prime degree; these allow us to prove Lemma \ref{pq}.  In Section
\ref{main-section}, we use Proposition \ref{theta} and Lemma \ref{pq}
to complete the proof of Theorem \ref{main-thm}.

{\bf Acknowledgments.}  We would like to thank 
Marston Conder and Gareth Jones for many very helpful conversations about
permutation groups of prime degree. 

\section{Proof of Theorem \ref{example-thing}}

The proof of Theorem \ref{example-thing} is quite simple.

\begin{proof}[Proof of Theorem \ref{example-thing}]

  Since $p$ divides $\phi(n)$, there is an $i$ in $(\Z/n\Z)^*$ such
  that $i$ has order $p$ in $(\Z/n\Z)^*$.  Then $(1 \cdots n)^i$ is
  also an $n$-cycle, so there is an element
  $\sigma$ in $S_n$ of order $p$ such that $\sigma(1) = 1$ and
  $\sigma (1 \, \cdots \, n) \sigma^{-1} = (1 \, \cdots \, n)^i$.  Let
  $\tau$ denote the cycle $(1 \, \cdots \, n)$.  Let $G = \langle
  \tau, \sigma \rangle$, let $N_1 = \langle \tau \rangle$ , and let
  $N_2 = \langle \sigma, \tau^p \rangle$.  Then $N_1$ and $N_2$ are
  both normal of index $p$ in $G$.  Since
  $|N_2| = n$ and a non-identity element of $N_2$ has a fixed point,
  $N_2$ cannot be transitive.

  \end{proof}

\section{Wreath Products and Orbits of Subgroups}\label{wreath}
\begin{lemma}\label{orb}
  Let $G$ be a transitive subgroup of $S_n$ and let $N$ be a normal subgroup of $G$.  Then, for any $a \in \{1, \dots, n \}$, we have $g N a = Nb$ for some $b \in \{1, \dots, n \}$. In particular, for any $a, b \in \{1, \dots, n \}$, we have $|Na| = |Nb|$.
\end{lemma}
\begin{proof}
  Let $ga = b.$ Since $N$ is normal in $G$, we have $gN = Ng$, so
  $Nb = Nga$ is simply all elements of $\{1, \dots, n \}$ of the form
  $gn a$ for some $n \in N$.  This set clearly has $|Na|$ elements.
\end{proof}

\begin{definition}
For a finite set $S$, we let $\Per(S)$ denote the group of bijections
from $S$ to itself (often called the permutations of $S$).  
\end{definition}

As usual, we let $S_m$ denote $\Per( \{ 1, \dots, m \})$.  We now
introduce wreath products. We follow the treatment of \cite{Odoni} and
\cite[Chapter 1.4]{Nek}; for a more general discussion of wreath products, see \cite[Page
172]{Rotman}).   

\begin{definition}
  We define the wreath product $S_m \wr S_n$ to be subgroup of
  \[ \Per( \{1, \dots, m \} \times \{1, \dots, n\}) \]
  consisting of all
  $\sigma \in \Per( \{1, \dots, m \} \times \{1, \dots, n\}$ such that
  for any $j, j' \in  \{1, \dots, n\}$, we have $\rho_1(\sigma(i,j)) =
  \rho_1(\sigma(i,j'))$ where $\rho_1: (k_1, k_2) \mapsto k_1$ for
  $(k_1, k_2) \in   \{1, \dots, m \} \times \{1, \dots, n\}$.  
\end{definition}

The wreath product $S_m \wr S_n$ may also be thought of as elements of
the form $(\sigma; \tau_1, \dots, \tau_n)$, with $\sigma \in S_m$ and
$\tau_k \in S_n$,  acting on the set of $(i,j)
\in \{1, \dots, m \} \times \{ 1, \dots, n \}$ by
\[ (\sigma; \tau_1, \dots, \tau_m))(i,j) = (\sigma(i), \tau_i(j)).\]

We let $p_1: S_m \wr S_n \lra S_m$ denote the homomorphism  sending $(\sigma; \tau_1, \dots, \tau_m)$
to $\sigma$.  (Note that the set of elements of the form $(e; \tau_1,
\dots, \tau_m)$ forms a normal subgroup of  $S_m \wr S_n$).

We let $\pi_i: S_m \wr S_n \lra S_n$ denote the map sending
$(\sigma; \tau_1, \dots, \tau_m)$ to $\tau_i$.  Note that while $\pi_i$ is
not a homomorphism in general on all of $S_m \wr S_n$, it does
restrict to a homomorphism on $\ker(p_1)$.  

\begin{prop}\label{theta}
Let $p < q$ be primes.  Let $G$ be a transitive subgroup of $S_{pq}$,
let $N_1$ be a transitive normal subgroup of $G$, let $N_2$ be a
non-transitive subgroup of $G$, and suppose that $|N_1| = |N_2|$.
Then there is a one-one homomorphism $\theta: G \lra S_p \wr S_q$ such
that
\begin{enumerate}
\item[(i)] $\theta(N_1)$ acts transitively on $\{ 1, \dots, p \} \times \{1,
\cdots q \}$;
\item[(ii)] $\theta(N_2) \subseteq \ker(p_1)$; and
 \item[(iii)] $\pi_i(\theta(N_2))$ is a transitive subgroup of $S_q$
   for all $i$.  
 \end{enumerate}
\end{prop}

\begin{proof}
  By Lemma \ref{orb}, the orbits of $N_2$ all have length $q$ of
  length $p$.  If they have length $p$ then $|N_2|$ divides a power of
  $p!$ and thus $q$ does not divide $|N_2|$.  But then $pq$ does not
  divide $|N_1|$ so $N_1$ cannot act transitively on
  $\{ 1, \dots, pq \}$.  Thus, every orbit of $N_1$ must have length
  $q$.  After renumbering, we may assume that each orbit of $N_1$ has
  the form $\{ (i-1)q + 1 \dots iq \}$ for $i = 1, \dots, p$.  Let
  $\iota : \{1, \dots, m\} \times \{1, \dots, n \} \lra \{1, \dots, mn
  \}$ by $\iota(i,j) = (i-1)p + j$.  Then letting
  $\theta(\sigma(i,j) = \iota \circ \sigma \circ \iota^{-1}(i,j)$
  gives a homomorphism
  $$\theta: G \lra \Per(\{1, \dots, m\} \times \{1, \dots, n \}.$$
  Since any element of $G$ sends $\{ (i-1)q + 1, \dots, iq \}$ to a set
  of the form $\{ (i'-1)q + 1 \dots i'q \}$ for some $i'$, we see that
  $\rho_1(\theta(g)(i,j)) = \rho_1(\theta(g)(i,j'))$ for any
  $j, j' \in \{1, \dots, n \}$, so $\theta(G)$ is contained in
  $S_m \wr S_n$.  Since $N_2$ sends each set
  $\{ (i-1)q + 1 \dots iq \}$ to itself, we see that $\theta(N_2)$
  fixes the first coordinate of any $(i,j)$ so
  $\theta(N_2) \subseteq \ker p_1$.  Likewise, since $\iota$ sends
  each orbit of $N_2$ on to a set of the form
  $\{ (i,1), \dots, (i,q) \}$, we see that for each $i$, the group
  $\pi_i(N_2)$ acts transitively on the set $\{ 1, \dots, n \}$.
\end{proof}

We will also use the following simple criterion for deciding when a
subgroup of $S_m \wr S_n$ acts transitively on $\{1, \dots, m \}
\times \{1, \dots, n \}$.


\section{Permutation Groups of Prime Degree}\label{perm}

The following is a theorem due to Wielandt \cite{Wie} (see also
\cite[Kapitel V, Bemerkung 21.7]{Huppert} and \cite[Section
3]{Cameron}).

\begin{theorem}\label{W}
Let $A$ be a transitive simple subgroup of $S_q$ for $q$ a prime.
Let $N_{S_q}(A)$ denote the normalizer of $A$ in $S_q$.  Then 
$N_{S_q}(A)/A$ is cyclic of order dividing $q-1$.    
\end{theorem}

Next, we state a well-known theorem of Burnside.  For $q$ a prime, we let $\A_1(q)$
denote the group of affine linear transformations of $\F_q$.  M\"uller
\cite{Muller} gives a quick proof that every transitive subgroup of
$S_q$ that is not subgroup of $\A_1(q)$ must be doubly transitive.
The fact that doubly transitive permutation groups contain nonabelian
normal subgroups is proved in \cite[Chapter X, Theorem
XIII]{Burnside}.  

\begin{theorem}\label{Burn}
  Let $A$ be a transitive subgroup of $S_q$ for $q$ a prime.  Then either $A$ is
  a subgroup of $\A_1(q)$ or
  $A$ is doubly transitive and contains a normal subgroup $B$ that is
  nonabelian and simple.
\end{theorem}

\begin{remark}
A more precise description of the transitive subgroups of prime degree
can be found in work of Jones
\cite{Jones} (see also \cite{Guralnick}).  Note that this more precise
description depends on the classification of finite simple groups,
whereas the proof of Theorem \ref{W} does not.  
\end{remark}

\begin{lemma}\label{contain}
  Let $G$ be a subgroup of $S_q$ and let $H$ be a transitive normal
  subgroup of $G$ that is simple.  Then every nontrivial normal
  subgroup of $G$ contains $H$.  In particular, if $p$ is a prime such
  that $p \nmid [G:H]$, then the only normal subgroup $N$ of $G$ such
  that $p \mid [G:N]$ is $N = \{ e \}$.
\end{lemma}
\begin{proof}
  Let $N$ be a normal subgroup of $G$.  Since $N$ is normal subgroup
  of $G$, it follows that $H \cap N$ is a normal subgroup of $H$, so
  $H \cap N$ is either $\{ e \}$ or all of $N$. Thus, if $N$ does not
  contain $H$, then $N \cap H = \{e \}$.  This means that
  $nh n^{-1} h^{-1} \in N \cap H = \{ e \}$ for any $h \in H$ and
  $n \in N$.  It follows that every element of $N$ commutes with every
  element of $H$.  Since $H$ is transitive, it contains a $q$-cycle,
  and $q$-cycles commute only with powers of themselves, so $N$ must
  be $\{ e \}$.  Thus, if $N$ is any nontrivial normal subgroup of
  $G$, then $[G:N]$ divides $[G:H]$, and thus cannot be divisible by
  any prime that does not divide $[G:H]$.  
\end{proof}

\begin{prop}\label{trivial}
Let $p < q$ be primes with $p \nmid (q-1)$.  
If $A$ is a transitive subgroup of $S_p$, and $B$ is a normal subgroup
of  $A$ with $p \mid [A:B]$, then $B$ is trivial.
\end{prop}
\begin{proof}
  By Theorem \ref{Burn}, $A$ is either a subgroup of $\A_1(q)$, the
  group of affine linear transformations of $\F_q$ or $A$ is doubly
  transitive contains a normal subgroup that is simple, nonabelian,
  and transitive.  Since $\A_1(q)$ has order $q(q-1)$ and $p$ does not
  divide $q-1$, it follows immediately if $p \mid [A:B]$, then $B$ is
  trivial.

  Now, suppose that $A$ contains a normal subgroup $N$ that is simple
  and nonabelian.  Since $A$ is doubly transitive, $N$ must be
  transitive; to see this, note that if we have $\sigma(i) = k$ for
  $\sigma \in A$ and $\tau(i) = i$, $\tau(k) = j$, then
  $\tau \sigma \tau^{-1}(i) = j$, for $i,j,k \in \{ 1, \dots q \}$.
  Applying Theorem \ref{W}, we see that $[A:N]$ divides $q-1$.  Thus,
  Lemma \ref{contain} implies any normal subgroup $B$ of $A$ such that
  $p \mid [A:B]$ must be trivial, as desired.
  
\end{proof}

The following is now an immediate consequence of Proposition
\ref{trivial} since for any transitive subgroup $A$ of $S_q$, we must
have $q \mid |A|$.  

\begin{lemma}\label{pq}
Let $p < q$ be primes with $p \nmid (q-1)$.  
If $A$ is a transitive subgroup of $S_p$, and $B$ is a normal subgroup
of  $A$ with $p \mid [A:B]$, then $q \mid [A:B]$.
\end{lemma}

\section{Proof of Theorem \ref{main-thm}} \label{main-section}

\begin{lemma}\label{product-index}
  Let $A \subseteq \prod_{i=1}^n G_i$ for some finite groups
  $G_i$.  Let $B$ be a normal subgroup of $A$.  Let $\pi_i$ denote the
  projection from $\prod_{i=1}^n G_i$ to its $i$-th component.  Then
  there are normal subgroups $M_i$ of $\pi(A_i)$, for $i=1, \dots, n$
  such that
  \begin{equation}\label{index}
  [A:B] = \prod_{i=1}^n [\pi_i(A): M_i]
  \end{equation}
\end{lemma}
\begin{proof}
  We proceed by induction.  The base case $n=1$ is obvious.  Now if $A \subseteq \prod_{i=1}^n G_i$ for some $n > 1$, then
\begin{equation}\label{indexab}
  [A:B] = [\ker(\pi_n) \cap A : \ker(\pi_n) \cap B]  \cdot [\pi_n(A): \pi_n(B)].
\end{equation}
Since $\ker(\pi_n) \cap A$ can be identified with a subgroup of $\prod_{i=1}^{n-1} G_i$, we may apply the inductive hypothesis to obtain
\begin{equation}\label{index3}
  [\ker(\pi_n) \cap A : \ker(\pi_n) \cap B]  =  \prod_{i=1}^{n-1} [\pi_i(A): M_i]
\end{equation}
for some normal subgroups $M_i$ of $\pi_i(A)$ for $i = 1, \dots, n-1$.
Combining \eqref{indexab} and \eqref{index3} gives \eqref{index}.  

\end{proof}

\begin{lemma}\label{index2}
  Let $A \subseteq \prod_{i=1}^n S_q$.  Let $B$ be a normal subgroup
  of $A$.  Let $\pi_i$ denote the projection from $\prod_{i=1}^n G_i$
  to its $i$-th component.  Suppose that $\pi_i(A)$ is a transitive
  subgroup of $S_q$ for all $i$.  Let $p < q$ be a prime that does not
  divide $q-1$.  Then if $p \mid [A:B]$, we also have $q \mid [A:B].$
\end{lemma}
\begin{proof}
  By Lemma \ref{index}, there are normal subgroups $M_i$ of
  $\pi(A_i)$, for $i=1, \dots, n$ such that
  $[A:B] = \prod_{i=1}^n [\pi_i(A): M_i]$.  By Lemma \ref{pq},
  for each $i$, we have $p \mid [\pi_i(A): M_i]$ whenever
  $q \mid [\pi_i(A): M_i]$.  Thus, if $p \mid [A:B]$, then $p$
  divides some $[\pi_i(A): M_i]$, which means that $q$ divides some
  $[\pi_i(A): M_i]$ and thus divides $[A:B]$, as desired.  
  
\end{proof}

\begin{proof}[Proof of \ref{main-thm}]
  Suppose that there is a subgroup $G$ of $S_{pq}$ that contains a
  transitive subgroup $N_1$ and a non-transitive subgroup $N_2$ such
  that $G/N_1 \cong G/N_2$.  By Lemma \ref{theta}, there is an
  embedding $\theta: G \lra S_p \wr S_q$ satisfying conditions
  (i)--(iii) of the lemma.  Let $M_2 = \ker p_1 \cap \theta(G)$; the
  $M_2$ is normal in $\theta(G)$ and contains $N_2$.  Since
  $\theta(G)/\theta(N_1) \cong \theta(G)/\theta(N_2)$, there is a
  normal subgroup $M_1$ of $\theta(G)$ that is normal in
  $\theta(G)/M_1 \cong \theta(G)/M_2$. Since
  $M_1 = \ker p_1 \cap \theta(G)$, we have
  $|M_1| = |M_2 \cap M_1| \cdot |\pi(M_1)|$.  Since $|M_1| = |M_2|$,
  this means that $[M_1: M_2 \cap M_1 ] = p_1 (M_1)$.  Now, $p_1(M_1)$
  must be a transitive subgroup of $S_p$, because $M_1$ is transitive
  on $\{ 1, \dots, p \} \times \{ 1, \dots q \}$.  Therefore, we have
  that $p$ divides $[M_1: M_2 \cap M_1 ]$ which in turns divides
  $|S_p| = p!$.  Now, $M_1$ is isomorphic to a subgroup $H$ of $S_q^p$
  such that $\pi_i(H)$ is transitive for each $i$.  Since
  $M_1 \cap M_2$ is normal in $M_1$, it follows from \ref{index2} that
  if since $p$ divides $[M_1: M_1 \cap M_2]$, then do does $q$.  But
  $[M_1: M_1 \cap M_2]$ divides $p!$ which is not divisible by $q$ so
  we have a contradiction.

\end{proof}

\providecommand{\MR}{\relax\ifhmode\unskip\space\fi MR }
\providecommand{\MRhref}[2]{%
  \href{http://www.ams.org/mathscinet-getitem?mr=#1}{#2}
}
\providecommand{\href}[2]{#2}

\end{document}